\documentclass{article}
\usepackage[utf8]{inputenc}
\usepackage[margin=1in]{geometry} 
\usepackage{amsmath,amsthm,amssymb}
\usepackage{thmtools}
\usepackage{graphicx}
\usepackage[ruled,vlined]{algorithm2e}

\usepackage{hyperref}
\usepackage{breakurl}
\usepackage{cleveref}
\usepackage{float}
\usepackage[utf8]{inputenc}
\usepackage[english]{babel}
\usepackage{biblatex}
\usepackage{tikz}
\usetikzlibrary{arrows,positioning,fit,shapes}

\newcommand{\R}{{\mathbb{R}}}

\SetKwComment{Comment}{/* }{ */}

\declaretheorem[numberwithin=section]{theorem}
\declaretheorem[sibling=theorem]{lemma}
\declaretheorem[sibling=theorem]{corollary}
\declaretheorem[sibling=theorem]{definition}
\declaretheorem[sibling=theorem]{fact}
\declaretheorem[sibling=theorem]{observation}

\crefname{theorem}{Theorem}{Theorems}
\crefname{lemma}{Lemma}{Lemmas}
\crefname{corollary}{Corollary}{Corollaries}
\crefname{definition}{Definition}{Definitions}
\crefname{fact}{Fact}{Facts}
\crefname{observation}{Observation}{Observations}
\crefname{section}{Section}{Sections}

\title{Graph Coloring and Semidefinite Rank\thanks{Supported by NSF grant CCF-2007009.}}

\author{Renee Mirka \and Devin Smedira \and David P.\ Williamson}

\date{%
\{rem379, dts88, davidpwilliamson\}@cornell.edu\\%
\vspace{3mm}
Cornell University, Ithaca NY 14850, USA}

\begin{document}

\maketitle

\begin{abstract}
This paper considers the interplay between semidefinite programming, matrix rank, and graph coloring.  Karger, Motwani, and Sudan \cite{KMS98} give a vector program for which a coloring of the graph can be encoded as a semidefinite matrix of low rank.  By complementary slackness conditions of semidefinite programming, if an optimal dual solution has sufficiently high rank, any optimal primal solution must have low rank.  We attempt to characterize graphs for which we can show that the corresponding dual optimal solution must have sufficiently high rank.  In the case of the original Karger, Motwani, and Sudan vector program, we show that any graph which is a $k$-tree has sufficiently high dual rank, and we can extract the coloring from the corresponding low-rank primal solution.  We can also show that if the graph is not uniquely colorable, then no sufficiently high rank dual optimal solution can exist.  This allows us to completely characterize the planar graphs for which dual optimal solutions have sufficiently high dual rank, since it is known that the uniquely colorable planar graphs are precisely the planar 3-trees.

We then modify the semidefinite program to have an objective function with costs, and explore when we can create a cost function whose optimal dual solution has sufficiently high rank.  We show that it is always possible to construct such a cost function given the graph coloring.  The construction of the cost function gives rise to a heuristic for graph coloring which we show works well in the case of planar graphs; we enumerated all maximal planar graphs with a $K_4$ of up to 14 vertices, and the heuristics successfully colored 99.75\% of them.

Our research was motivated by the Colin de Verdi\`ere graph invariant \cite{CDV90} (and a corresponding conjecture of Colin de Verdi\`ere), in which matrices that have some similarities to the dual feasible matrices must have high rank in the case that graphs are of a certain type; for instance, planar graphs have rank that would imply the 4-colorability of the primal solution.  We explore the connection between the conjecture and the rank of the dual solutions.
\end{abstract}

\section{Introduction}

Given an undirected graph $G=(V,E)$, a {\em coloring} of $G$ is an assignment of colors to the vertices $V$ such that for each edge $(i,j) \in E$, $i$ and $j$ receive different colors.  The {\em chromatic number} of $G$, denoted $\chi(G)$, is the minimum number of colors used such that a coloring of $G$ exists.  The {\em clique number} of a graph $G$, denoted $\omega(G)$, is the size of the largest {\em clique} in the graph; a set $S \subseteq V$ of vertices is a clique if for every distinct pair $i,j \in S$, $(i,j) \in E$.  It is easy to see that $\omega(G) \leq \chi(G)$.  Graph colorings have been intensively studied for over a century.  One of the most well-known theorems of graph theory, the {\em four-color theorem}, states that four colors suffice to color any planar graph $G$; the problem of four-coloring a planar graph can be traced back to the 1850s, and the computer-assisted proof of the four-color theorem by Appel and Haken \cite{AppelH77, AppelHK77} is considered a landmark in graph theory.  See Jensen and Toft \cite{JensenT95} and Molloy and Reed \cite{MolloyR02} for book-length treatments of graph coloring in general. Fritsch and Fritsch \cite{FritschF94}, Ore \cite{Ore67}, and Wilson \cite{Wilson02} provide book-length treatments of the four-color theorem in particular, and Robertson, Sanders, Seymour, and Thomas \cite{RobertsonSST97} give a simplified computer-assisted proof of the four-color theorem. 

This paper considers the use of semidefinite programming in graph coloring.  The connection between semidefinite programming and graph coloring was initiated by Lov\'asz \cite{Lovasz79}, who introduced the Lov\'asz theta function, $\theta(\bar G)$, which is computable via semidefinite programming; $\bar G$ is the complement of graph $G$, in which all edges of $G$ are replaced by nonedges and vice versa.  Lov\'asz showed that $\omega(G) \leq \theta(\bar G) \leq \chi(G)$; a helpful overview of this result was given by Knuth \cite{Knuth94}.  

Another use of semidefinite programming for graph coloring was introduced by Karger, Motwani, and Sudan \cite{KMS98} (KMS), who showed how to color $k$-colorable graphs with $O(n^{1-3/(k+1)}\log^{1/2}n)$ colors in polynomial time using semidefinite programming, where $n$ is the number of vertices in the graph.  A starting point of the algorithm of KMS is the following vector program, which KMS called the {\em strict vector chromatic number}; the vector program can be solved via semidefinite programming:
\begin{equation*}
    \begin{array}{llll}
        \text{minimize} & \alpha &\\
        \text{subject to} & v_i \cdot v_j = \alpha, & \forall (i,j)\in E,\\
                          & v_i \cdot v_i = 1,    & \forall i \in V, \\
                          & v_i \in \mathbb{R}^n, & \forall i \in V.
    \end{array}
\end{equation*}
KMS observe that any $k$-colorable graph has a feasible solution to the vector program with $\alpha = -1/(k-1)$: let $v_1 = (1, 0, \dots, 0) \in \Re^{k-1}$ and inductively find $v_i \in \Re^{k-1}$ for $1 < i \leq k-1$ by setting $v_i(j) = 0$ for $j>i$ and otherwise solving the system of equations given by $v_l \cdot v_i = -1/(k-1)$ for $1\leq l \leq i-1$ and $v_i \cdot v_i = 1$. Finally, let $v_k = -\sum_{j=1}^{k-1}v_j$, then assign one color to each of the vectors $v_i$. This guarantees that each vector $v_i$ has unit length (so $v_i \cdot v_i = 1$) and that for any edge $(i,j) \in E$, $v_i \cdot v_j = -1/(k-1)$. 
It is important for the following discussion to observe that this solution lies in a $(k-1)$-dimensional space.  KMS also observe that there is a natural connection between the strict vector chromatic number and the Lov\'asz theta function.  In particular, for the solution $\alpha$ to the vector program above, it is possible to show that $\alpha = 1/(1-\theta(\bar G))$ (see \cite[Theorem 8.2]{KMS98}).  If the graph $G$ has a $k$-clique $K_k$ and is $k$-colorable, then by Lov\'asz's theorem, $\theta(\bar G)=k$, and so the feasible solution with $\alpha=-1/(k-1)$ is an optimal solution.  It is also possible to argue directly that a graph with a $K_k$ must have $\alpha \geq -1/(k-1)$, again proving that the feasible solution given above is an optimal one.  We will call the feasible solution above (in which all the vectors are recursively constructed) the {\em reference solution}.  

The goal of this paper is to explore situations in which the reference solution is the unique optimal solution of a semidefinite program (SDP), either the SDP corresponding to the strict vector chromatic number given above, or another that we will give shortly.  To do this, we will use complementary slackness conditions for semidefinite programs.  Consider the primal and dual SDPs shown in standard form below, where the constraint that $X$ is a positive semidefinite matrix is represented by $X \succeq 0$, and we take the outer product of matrices, so that $C \bullet X$, for instance, denotes $\sum_{i=1}^\ell \sum_{j=1}^\ell c_{ij} x_{ij}$.
\begin{equation*}
    \begin{array}{llllll}
        \text{minimize} & C \bullet X&  & \text{maximize}&  b^T y\\
        \text{subject to} & A_i \bullet X = b_i & \mbox{for } i = 1, \ldots, m, \qquad & \text{subject to} & S = C - \sum_{i=1}^m y_i A_i, \\
        & X \succeq 0, & & & S \succeq 0, \\
        & X \in  \Re^{\ell \times \ell}, & & & S \in \Re^{\ell \times \ell}.
    \end{array}
\end{equation*}
Duality theory for semidefinite programs (e.g. Alizadeh \cite{Alizadeh95}) shows that for any feasible primal solution $X$ and any feasible dual solution $y$, $C \bullet X \geq b^T y$.  Furthermore if $C \bullet X = b^Ty$, so that the solutions are optimal, then it must be the case that $\mbox{rank}(X) + \mbox{rank}(S) \leq \ell$, and $XS = 0$. Thus if we want to show that any optimal primal solution has rank at most $r$, it suffices to show the existence of an optimal dual solution of rank at least $\ell -r$.  Turning back to the strict vector chromatic number vector program, the corresponding dual vector program is 
\begin{equation*}
    \begin{array}{llll}
        \text{maximize} & -\sum_i v_i\cdot v_i &\\
        \text{subject to} & \sum_{i\neq j} v_i \cdot v_j \geq 1, &  \\
                          & v_i \cdot v_j = 0,    & \forall (i,j) \notin E, i\neq j \\
                          & v_i \in \mathbb{R}^n, & \forall i \in V.
    \end{array}
\end{equation*}
Thus, given a $k$-colorable graph $G$ with a $K_k$, if we can show a dual feasible solution of value $-1/(k-1)$ and rank $n-k+1$, then we know that the primal solution must have rank at most $k-1$; we'll show that this will give us the reference solution\footnote{There are some subtleties here we are glossing over in the interest of getting across the main idea.  In particular, the SDP corresponding to the strict vector chromatic number vector program has dimension $n+1$, not $n$; we explain why and why that doesn't matter for our purposes in Section \ref{sec:svcn}.}.  We will for shorthand say that there is an optimal dual of {\em sufficiently high rank}.

Our first result is to partially characterize the set of graphs for which the optimal solution to the strict vector chromatic number vector program  is the reference solution.  In particular, we can show that if the graph is a {\em $(k-1)$-tree}, then the reference solution is the unique optimal solution to the SDP.  In the opposite direction, if the graph is not {\em uniquely colorable}, then the dual does not have sufficiently high rank, and there exist optimal primal solutions that are not the reference solution and are at least $k$-dimensional.  A $(k-1)$-tree is a graph constructed by starting with a complete graph on $k$ vertices.  We then iteratively add vertices $v$; for each new vertex $v$, we add $k-1$ edges from $v$ to previously added vertices such that $v$ together with these $k-1$ neighbors form a clique.  A $k$-colorable graph is uniquely colorable if it has only one possible coloring up to a permutation of the colors.  The $k$-tree graphs are easily shown to be unique colorable.  In the case of planar graphs with a $K_4$, these results imply a complete characterization of the graphs for which the optimal solution is the reference solution, since it is known that the uniquely 4-colorable planar graphs are exactly the planar 3-trees, also known as the Apollonian networks \cite{Fowler98}.    We argue that it is not surprising that graphs are not uniquely $k$-colorable do not have the reference solution as the sole optimal solution; we show that one can find a convex combination of the two different reference solutions corresponding to the two different colorings that gives an optimal SDP solution of rank higher than $k-1$, and clearly the convex combination is also feasible for the SDP.

To get around the issue of unique colorability, we instead look for minimum-cost feasible solutions to the SDP above.  That is, given a cost matrix $C$, we look to find optimal solutions to the primal SDP
\begin{equation*}
    \begin{array}{llll}
        \text{minimize} & C \bullet X &\\
        \text{subject to} & X_{ij} = -1/(k-1), & \forall (i,j)\in E \\
          & X_{ii} = 1,    & \forall i \in V \\
         & X \succeq 0.& 
    \end{array}
\end{equation*}
\noindent The corresponding dual SDP is 
\begin{equation*}
    \begin{array}{llll}
        \text{maximize} & \sum_{i=1}^n y_i - \frac{2}{k-1} \sum_{e \in E}z_e &\\
        \text{subject to} & S = C - \sum_{i=1}^n y_iE_{ii} - \sum_{e \in E} z_e E_e &  \\
          & S\succeq 0,& 
    \end{array}
\end{equation*}
where $E_{ii}$ is the matrix with a 1 at position $ii$ and 0 elsewhere and for $e=(i,j)$, $E_e$ is the matrix with 1 at positions $ij$ and $ji$ and 0 elsewhere.
Once again, the reference solution is a feasible solution to the primal SDP.  The goal now is to find a cost function $C$ such that there is an optimal dual solution of sufficiently high rank (here rank $n-k+1$), so that the reference solution is the unique optimal solution to the primal SDP.  We show that it is always possible to find a cost function $C$ such that the dual has sufficiently high rank.  Our construction of $C$ depends on the coloring of the graph; however, we do show that such a $C$ exists.  

Furthermore, the construction of $C$ suggests a heuristic for finding a coloring of the graph, and we show that the heuristic works well.   We enumerated all maximal planar graphs of up to 14 vertices containing a $K_4$.  The heuristics successfully colored all graphs of up to 11 vertices, and at least 99.75\% of all graphs on 12, 13, and 14 vertices.    The heuristics involve repeatedly solving semidefinite programs, and thus are not practical for large graphs (although they still run in polynomial time).  However, we view them as a proof of concept that it might be possible to use our framework to reliably 4-color planar graphs.

Our interest in this direction of research was prompted by the Colin de Verdi\`ere invariant \cite{CDV90} (see also \cite{HolstLS99} for a useful survey of the invariant).  A {\em generalized Laplacian} $L = (\ell_{ij})$ of graph $G$ is a matrix such that the entries $\ell_{ij} < 0$ when $(i,j) \in E$, and $\ell_{ij} = 0$ when $(i,j) \notin E$.  The Colin de Verdi\'ere invariant, $\mu(G)$, is defined as follows.
\begin{definition}  The Colin de Verdi\`ere invariant $\mu(G)$ is the largest corank of a generalized Laplacian $L$ of $G$ such that:
\begin{enumerate}
    \item $L$ has exactly one negative eigenvalue of multiplicity one;
    \item there is no nonzero matrix $X=(x_{ij})$ such that $LX=0$ and such that $x_{ij} = 0$ whenever $i=j$ or $\ell_{ij} \neq 0$.
\end{enumerate}
\end{definition}
\noindent Colin de Verdi\'ere shows that  $\mu(G) \leq 3$ if and only if $G$ is planar; in other words, {\em any} generalized Laplacian of $G$ with exactly one negative eigenvalue of multiplicity 1 will have rank at least $n-3$ (modulo the second condition on the invariant, which we will ignore for the moment).  Other results show that $G$ is outerplanar if and only if $\mu(G) \leq 2$, and $G$ is a collection of paths if and only if $\mu(G) \leq 1$.  Colin de Verdi\`ere \cite{CDV90} conjectures that $\chi(G) \leq \mu(G) + 1$; this result is known to hold for $\mu(G) \leq 4$.  We note if $G$ is planar that the part of the dual matrix $ - \sum_{i=1}^n y_iE_{ii} - \sum_{e \in E} z_e E_e$ is indeed a generalized Laplacian $L$ of a planar graph when the $z_e \geq 0$ for all $e \in E$, and that if $G$ is connected, then the $y_i$ can be adjusted so that this matrix has a single negative eigenvalue of multiplicity one.  Thus this part of the matrix, under these conditions, must have sufficiently high rank, as desired to verify that the optimal primal solution is the reference solution.  This would show that if the graph $G$ has a clique on $\mu(G)+1$ vertices, then indeed $\chi(G) = \mu(G)+1$.  So, for example, this would prove that any planar graph with a $K_4$ can be four-colored, leading to a non-computer assisted proof of the four-color theorem. However, we do not know how to find the corresponding cost matrix $C$ or show that the dual $S$ we find is optimal.   Still, we view our heuristics as a step towards finding a way to construct the cost matrix $C$ without knowledge of the coloring, and without reliance on the machinery of the proofs of the four-color theorem that have been developed thus far.  

The rest of this paper is structured as follows.  In Section \ref{prelim}, we give some preliminary results on semidefinite programming.  In Section \ref{sec:svcn}, we show our results for the strict vector chromatic number SDP, and show that $(k-1)$-trees imply dual solutions of sufficiently high rank, while graphs that are not uniquely colorable imply that such dual solutions cannot exist.  In Section \ref{modify}, we turn to the SDP with cost matrix $C$, and show that for any $k$-colorable graph with a $k$-clique, a cost matrix $C$ exists that gives rise to a dual of sufficiently high rank.  In Section \ref{experiment}, we give two heuristics for coloring planar graphs based on our construction of the cost matrix $C$, and show a case where the heuristic fails to find a 4-coloring of a planar graph.  Finally, we turn to some further thoughts and remaining open questions in Section \ref{conc}.


\section{Preliminaries}\label{prelim}

In this section, we recall some basic facts about semidefinite matrices and semidefinite programs that we will use in subsequent sections.

As noted in the introduction, the primal and dual semidefinite programs, which we will label by $(P)$ and $(D)$ respectively, are as follows:
\begin{equation*}
    \begin{array}{llllll}
        \text{minimize} & C \bullet X&  & \text{maximize}&  b^T y\\
        \text{subject to} & A_i \bullet X = b_i, & \mbox{for } i = 1, \ldots, m, \qquad & \text{subject to} & S = C - \sum_{i=1}^m y_i A_i, \\
        (P) & X \succeq 0, & & (D) & S \succeq 0 ,\\
        & X \in  \Re^{\ell \times \ell}, & & & S \in \Re^{\ell \times \ell}.
    \end{array}
\end{equation*}
We use $\bullet$ to denote matrix outer products, so that $C \bullet X$, for instance, denotes $\sum_{i=1}^\ell \sum_{j=1}^\ell c_{ij} x_{ij}$.

We always have weak duality for semidefinite programs, so that the following holds.
\begin{fact}  Given any feasible $X$ for $(P)$ and $y$ for $(D)$, $C \bullet X \geq b^Ty$.
\end{fact}
Thus if we can produce a feasible $X$ for $(P)$ and a feasible $y$ for $(D)$ such that $C \bullet X = b^Ty$, then $X$ must be optimal for $(P)$ and $y$ optimal for $(D)$.

The following is also known, and is the semidefinite programming version of complementary slackness conditions for linear programming.
\begin{fact}{\cite[Theorem 2.10, Corollary 2.11]{Alizadeh95}} \label{fact:sdp-dual-rank}
For optimal $X$ for $(P)$ and $y$ for $(D)$, $XS = 0$ and $\text{rank}(X) + \text{rank}(S) \leq \ell$.
\end{fact}

Semidefinite programs and vector programs (such as the strict vector chromatic vector program) are equivalent because a symmetric $X \in \Re^{n \times n}$ is positive semidefinite if and only if $X = QDQ^T$ for a real matrix $Q \in \Re^{n\times n}$ and diagonal matrix $D$ in which the entries of $D$ are the eigenvalues of $X$, and the eigenvalues are all nonnegative.  We can then consider $D^{1/2}$, the diagonal matrix in which each diagonal entry is the square root of the corresponding entry of $D$.  Then $X = (QD^{1/2})(QD^{1/2})^T$.  If we let $v_i \in \Re^n$ be the $i$th row of $QD^{1/2}$, then $x_{ij} = v_i \cdot v_j$, and similarly, given the vectors $v_i$, we can construct a semidefinite matrix $X$ with $x_{ij} = v_i \cdot v_j$.  We also make the following observation based on this decomposition.
\begin{observation} \label{obs:dim}
Given a semidefinite matrix $X=QDQ^T \in \Re^{n \times n}$, $\text{rank}(X) = d$ if and only if the vectors $v_i \in \Re^n$ with $v_i$ the $i$th row of $QD^{1/2}$ are supported on just $d$ coordinates.
\end{observation}

\section{The Strict Vector Chromatic Number SDP}
\label{sec:svcn}

Recall the strict vector chromatic SDP given in the introduction, which we now label (SVCN):
\begin{equation*}
    \begin{array}{llll}
        \text{minimize} & \alpha &\\
        \text{subject to} & v_i \cdot v_j = \alpha, & \forall (i,j)\in E, \\
            \text{(SVCN)}         & v_i \cdot v_i = 1,    & \forall i \in V, \\
                          & v_i \in \mathbb{R}^n, & \forall i \in V.
    \end{array}
\end{equation*}
In this section, we give a partial characterization of graphs for which the reference solution is an optimal solution to (SVCN).

First, we observe that (SVCN) is equivalent to the following semidefinite program:
\begin{equation*}
    \begin{array}{llll}
        \text{minimize} & - z_{00} &\\
        \text{subject to} & z_{ij} + z_{00} = 0, & \forall (i,j)\in E \\
              (SVCN\mbox{-}P)            & z_{ii} = 1,    & \forall i \in V \\
                         & z_{i0} = z_{0i} = 0 & \forall i \in V \\
                          & Z=(z_{ij}) \succeq 0, \\
                          & Z\text{ symmetric, } Z \in \Re^{(n+1)\times (n+1)}
    \end{array}
\end{equation*}
The dual of this SDP is 
\begin{equation*}
    \begin{array}{llll}
        \text{maximize} & - \sum_{i \in V} w_{ii} &\\
        \text{subject to} & w_{00} = -1 + \sum_{i\neq j} w_{ij}, \\
                   (SVCN\mbox{-}D)      
                          & w_{ij} = 0,         & \forall (i,j) \notin E\\
                          & W=(w_{ij}) \succeq 0, \\
                          & W \text{ symmetric, } W \in \Re^{(n+1)\times (n+1)}
    \end{array}
\end{equation*}

In what follows, we will want to relate the rank of the primal submatrix $X=(z_{ij})_{i, j \in V}$ to the rank of the dual submatrix $S = (w_{ij})_{i, j \in V}$; that is, we want to look at the submatrices that don't contain the 0th row and column of the primal solution (corresponding to the variable $\alpha$ in (SVCN)) and the corresponding 0th row and column of the dual solution.  We will henceforward in this section call these submatrices $X$ and $S$.

\begin{lemma}
Given an optimal primal solution $Z$ to (SVCN-P) and optimal dual solution $W$ to (SVCN-D), we have that $\text{rank}(X) + \text{rank}(S) \leq n$.
\end{lemma}

\begin{proof}
If for optimal dual solution $S$, the submatrix $S=(s_{ij})_{i \in V}$ has rank at least $n-k$, then $S$ has rank at least $n-k$.  Then by Fact \ref{fact:sdp-dual-rank}, any optimal primal solution $W$ to (SVCN-P) must have rank at most $(n+1) - (n-k) = k+1$.  By Observation \ref{obs:dim}, the dimension of the corresponding vectors $v$ of the matrix $W$ must be at most $k+1$.  But we note that by the condition that $w_{i0} = v_i \cdot v_0 = 0$ for all $i \in V$, it must be the case that all vectors $v_i$ for $i \in V$ are orthogonal to $v_0$, so that the vectors $v_i$ for $i \in V$ lie in dimension at most $k$.  Then by Observation \ref{obs:dim} the rank of $X$ is at most $k$, giving the desired inequality.  

Similarly, if the rank of $X$ is at least $k$, then because $w_{00}$ is positive, the rank of $Z$ must be at least $k+1$.  Then by Fact \ref{fact:sdp-dual-rank}, the rank of $W$ must be at most $(n+1)-(k+1)=n-k$, so that the rank of $S$ is at most $n-k$.
\end{proof}

Because the values of $Z$ and $W$ are determined by the submatrices $X$ and $S$, we will for the rest of the section refer to primal solutions $X$ and dual solutions $S$.

Our main result for this section is about graphs that are {\em $k$-trees}.

\begin{definition}
A $(k-1)$-tree with $n$ vertices is an undirected graph constructed by beginning with the complete graph on $k$ vertices and repeatedly adding vertices in such a way that each new vertex, $v$, has $k-1$ neighbors that, together with $v$, form a $k$-clique.
\end{definition} 

An easy inductive argument shows that these graphs are $k$-colorable.  Also, $(k-1)$-trees are known to be uniquely $k$-colorable, where uniquely colorable means every coloring produces that same vertex partitioning. Once $k$ colors are assigned to the initial complete graph with $k$ vertices, the color of each new vertex is uniquely determined by its $k-1$ neighbors. This partitioning into color classes is unique up to permuting the colors. Note that by construction, a $(k-1)$-tree contains a $K_{k}.$ 

Recall that Karger, Motwani, and Sudan \cite{KMS98} show that the solution to (SVCN) is $-1/(\vartheta(\bar{G}) - 1)$ where $\vartheta(\bar{G})$ is the Lov\'asz theta function. Lov\'asz \cite{Lovasz79} proved that $\omega(G)\leq \vartheta(\bar{G}) \leq \chi(G)$ where $\omega(G)$ and $\chi(G)$ are the clique and chromatic numbers of $G$ respectively. In particular, if a graph is $c$-colorable, the optimal solution to this vector program is at most $-1/(c-1)$. Note that as previously remarked, $(k-1)$-trees contain $K_{k}$ cliques and are $k$-colorable. As a result, the optimal value of (SVCN) for a $(k-1)$-tree will be exactly $-1/(k-1)$.

Our goal is to show there is a feasible solution to the dual of (SVCN) with high rank. In particular, given a $(k-1)$-tree with $n$ vertices, we show the existence of a dual solution with rank at least $n-k+1$. This ensures that any primal solution has rank at most $k-1$; we show that the reference solution is the unique optimal primal solution. This is formalized in the following theorem.

\begin{restatable}{theorem}{svcn}\label{result}
Given a $(k-1)$-tree $G$ with $n$ vertices, there is an optimal dual solution $S$ to (SVCN-D) with rank at least $n - k+1$, and thus any optimal primal solution $X$ to (SVCN-P) has rank at most $k-1$.
\end{restatable}

We subsequently prove that the reference solution is indeed the unique optimal solution in this case.

\begin{restatable}{theorem}{svcncolor}\label{thm:svcn-color} The reference solution is the unique optimal primal solution (up to rotation) for a $(k-1)$-tree $G =(V,E)$.
\end{restatable}

To prove Theorem \ref{result}, we need a number of supporting lemmas.  We begin with the following. 

\begin{lemma} \label{lem:ktreecount}
Let $tri(G)$ denote the number of triangles in a $(k-1)$-tree, $G$. Then, for a $(k-1)$-tree $G$ with $n$ vertices,
\begin{align}
      |E(G)| &= (2n-k)\frac{k-1}{2}  \label{eq:edges} \\
      tri(G) &= \frac{(3n-2k)(k-1)(k-2)}{6}. \label{eq:triangles}
\end{align}
\end{lemma}

\begin{proof}
We first prove \eqref{eq:edges} by induction. The smallest $(k-1)$-tree is the complete graph with $n = k$ vertices. This graph has ${k \choose 2} = \frac{k(k-1)}{2}$ edges. We also have $(2n-k)\frac{k-1}{2} = (2k-k)\frac{k-1}{2} = \frac{(k-1)k}{2}$. Assume the claim is true for all $(k-1)$-trees with at most $n$ vertices. If we add an $n+1st$ vertex, we are also adding $k-1$ new edges. Our new graph will have $\frac{(2n-k)(k-1)}{2} + (k-1) = \frac{(2n-k)(k-1) + 2(k-1)}{2} = \frac{(2(n+1)-k)(k-1)}{2}$ edges, as desired.\\

To count triangles, we begin with the complete graph on $k$ vertices again. This graph has ${k \choose 3} = \frac{k(k-1)(k-2)}{6}$ triangles. We also have $\frac{(3n-2k)(k-1)(k-2)}{6} = \frac{(3k-2k)(k-1)(k-2)}{6} = \frac{k(k-1)(k-2)}{6}.$  Assume the claim is true for all $(k-1)$-trees with at most $n$ vertices. If we add an $n+1st$ vertex, we are also adding ${k-1 \choose 2}$ new triangles. Then this new graph has $\frac{(3n-2k)(k-1)(k-2)}{6} + \frac{(k-1)(k-2)}{2} = \frac{(3n-2k)(k-1)(k-2) + 3(k-1)(k-2)}{6} = \frac{(3(n+1)-2k)(k-1)(k-2)}{6}$ triangles, as desired.
\end{proof}

Consider a $(k-1)$-tree $G$ with $n$ vertices. For $v\in V$ we denote the neighborhood of $v$ by $N(v) = \{u:(u,v)\in E\}.$ We define the following matrix $S(G)$ which may be referred to as $S$ if $G$ is clear from context.

\[ S(G)_{ij}= \begin{cases} 
      \dfrac{|N(i)| - (k-2)}{k(k-1)(n-k+1)} & i = j \\
       & \\
      \dfrac{|N(i) \cap N(j)| - (k-3)}{k(k-1)(n-k+1)} & (i,j) \in E  \\
      & \\
      0 & (i,j) \notin E, i \neq j.
   \end{cases}
\]

We will show that $S(G)$ is an optimal dual solution with rank $n-k+1$. First, we show $S(G)$ is a feasible solution with help from the following lemma. 

\begin{lemma}\label{PSD}
For a $(k-1)$-tree $G$ with $n$ vertices, $S(G)$ is positive semidefinite.
\end{lemma}

\begin{proof}
Observe that it suffices to show $S'(G) = k(k-1)(n-k+1)S(G)$ is positive semidefinite (PSD) since $k(k-1)(n-k+1) > 0$ for $n\geq k$. We proceed by induction. First consider $(k-1)$-trees with $k$ vertices. There is only one, $G = K_{k}$. Furthermore, $S'(K_{k})$ is equal to the all-ones matrix which has eigenvalues $k$ and 0 with multiplicity $k-1$ and thus is PSD.  

Now assume there is some integer $n$ such that for every $(k-1)$-tree, $G$, with at most $n$ vertices, $S'(G)$ is PSD.
Consider a $(k-1)$-tree $G$ with $n+1$ vertices. Since it is a $(k-1)$-tree, it can be constructed from some smaller $(k-1)$-tree $G'$ with $n$ vertices by adding a vertex $v$ and $(k-1)$ edges that form a $k$ clique with the $k-1$ neighbors. By assumption, $S'(G')$ is PSD. Let $I$ be the set of indices of the $k-1$ neighbors of $v$. Then we observe that 
$S'(G) = T + v_{n+1}v_{n+1}^T$ where 
$$ T = \begin{bmatrix}
   & &  &\vline & 0\\
    & S'(G') & &  \vline & \vdots \\
    & & &  \vline &  0\\
\hline
  0 & \cdots & 0 & \vline & 0
\end{bmatrix}$$
  and $$v_{n+1}(i) = \begin{cases} 1 & i\in I \cup \{n+1\} \\
  0 & \text{otherwise}
  \end{cases}.$$ \\
  Then $x^TS'(G)x = x^TTx + x^Tv_{n+1}v_{n+1}^Tx \geq x^Tv_{n+1}v_{n+1}^Tx = (v_{n+1}^Tx)^2 \geq 0$ where the first inequality is due to $T$ being PSD since $S'(G')$ is PSD.
\end{proof}

\begin{lemma}\label{feasible}
For a $(k-1)$-tree $G$ with $n$ vertices, $S(G)$ is a feasible dual slack matrix.
\end{lemma}

\begin{proof}
\cref{PSD} shows that $S(G)$ is PSD. To complete this claim, we must show that the dual constraints are satisfied. That $S(G)_{ij} = 0$ for $(i,j) \notin E$ is clear by construction. The other constraint requires $\sum_{i\neq j} s_{ij} \geq 1$. For $S(G)$,
$$ 
\begin{aligned}
\sum_{i\neq j} S_{ij} &= 2\sum_{(i,j) \in E} \dfrac{|N(i) \cap N(j)| - (k-3)}{k(k-1)(n-k+1)} \\
&= \frac{1}{k(k-1)(n-k+1)}\left[-2(k-3)|E(G)| + 2\sum_{(i,j)\in E} |N(i) \cap N(j)|\right] \\
&=\frac{1}{k(k-1)(n-k+1)}\left[-2(k-3)((2n-k)\frac{k-1}{2}) + 2\sum_{v \in V}\text{(\# of triangles in $G$ containing $v$)}\right] \\
&=\frac{1}{k(k-1)(n-k+1)}\Big[-(k-3)(k-1)(2n-k) + 6tri(G)\Big] \\
&= \dfrac{-(k-3)(k-1)(2n-k) + 6(\frac{(3n-2k)(k-1)(k-2)}{6})}{k(k-1)(n-k+1)} \\
&=\frac{(k-1)((k-2)(3n-2k)-(k-3)(2n-k))}{k(k-1)(n-k+1)}\\
&=\frac{(k-1)(3nk-2k^2-6n+4k-2nk+k^2+6n-3k)}{k(k-1)(n-k+1)}\\
&=\frac{k(k-1)(n-k+1)}{k(k-1)(n-k+1)} = 1
\end{aligned}
$$ where we use both \eqref{eq:edges} and \eqref{eq:triangles} from Lemma \ref{lem:ktreecount}.
\end{proof}

We can now show that $S(G)$ is an optimal dual solution.

\begin{theorem}\label{opt}
For a $(k-1)$-tree $G$ with $n$ vertices, $S(G)$ is an optimal dual solution.
\end{theorem}

\begin{proof}
We remarked earlier that the optimal primal solution for a $(k-1)$-tree is $-1/(k-1)$. Thus for $S(G)$ to be an optimal dual solution, it suffices to show that $-\sum_{i}S_{ii} = -1/(k-1)$. Again using \eqref{eq:edges} from Lemma \ref{lem:ktreecount}, we have
$$
\begin{aligned}
-\sum_{i=1}^nS_{ii} &= -\sum_{i=1}^n \frac{|N(i)|-(k-2)}{k(k-1)(n-k+1)} \\
&= -\frac{1}{k(k-1)(n-k+1)}\left[-(k-2)n + \sum_{i=1}^n |N(i)|\right] \\
&= -\dfrac{-(k-2)n + 2|E|}{k(k-1)(n-k+1)}\\
&= -\dfrac{-(k-2)n+((2n-k)(k-1))}{k(k-1)(n-k+1)}\\
&= -\dfrac{-nk+2n+2nk-2n-k^2+k}{k(k-1)(n-k+1)} = -1/(k-1).
\end{aligned}
$$
\end{proof}

Finally, we want to show that for a $(k-1)$-tree $G$ with $n$ vertices, $S(G)$ has rank at least $n-k+1$. This guarantees that any primal solution has rank at most $k-1$. 

\begin{theorem}\label{rank}
For a $(k-1)$-tree $G$ with $n$ vertices, $S(G)$ has rank at least $n-k+1$.
\end{theorem}

\begin{proof}
It again suffices to show the claim is true for $S'(G) = k(k-1)(n-k+1)S(G)$. Proceeding by induction, for $n=k$ we have $rank(S'(G)) = rank(S'(K_{k})) = 1 = k-(k-1)$ with 
$S'(K_{k})$ equal to the all-ones matrix.  Assuming the claim is true for all $(k-1)$-trees with at most $n$ vertices, we consider a $(k-1)$-tree $G$ with $n+1$ vertices. We again use the decomposition $S'(G) = T + v_{n+1}v_{n+1}^T$ where 
$$ T = \begin{bmatrix}
   & &  &\vline & 0\\
    & S'(G') & &  \vline & \vdots \\
    & & &  \vline &  0\\
\hline
  0 & \cdots & 0 & \vline & 0
\end{bmatrix},
  v_{n+1}(i) = \begin{cases} 1 & i \in I \cup \{n+1\} \\
  0 & \text{otherwise}
  \end{cases},$$
  and $G'$ is a $(k-1)$-tree with $n$ vertices acquired by removing vertex $n+1$ with exactly $k-1$ neighbors, $i\in I$, from $G$. Note $dim(ker(T)) = dim(ker(S'(G')) + 1 \leq k$ by assumption. Now assume $x \in ker(S'(G))$. Then 
  $$0 = x^TS'(G)x = x^TTx + x^Tv_{n+1}v_{n+1}^Tx.$$ Since $T$ and $v_{n+1}v_{n+1}^T$ are both PSD, this implies $x^TTx = 0$ and $x^Tv_{n+1}v_{n+1}^Tx = 0$. Therefore $ker(S'(G)) = ker(T) \cap ker(v_{n+1}v_{n+1}^T)$. However, note that $x = (0, \cdots, 0, 1) \in ker(T)$, but $x \notin ker(v_{n+1}v_{n+1}^T)$. Then 
  $$ker(S'(G)) = ker(T) \cap ker(v_{n+1}v_{n+1}^T) \subsetneq ker(T).$$ This implies $dim(ker(S'(G)) < dim(ker(T)) \leq k$, so $rank(S'(G)) \geq (n+1)-k+1$.
\end{proof}

\svcn*
\begin{proof}
\cref{result} follows as an immediate consequence of \cref{feasible}, \cref{opt}, and \cref{rank}.
\end{proof}

We now turn to showing that the reference solution is indeed the optimal solution in the case of $(k-1)$-trees.

\svcncolor*

\begin{proof}
We already know that the reference solution given by $u_1, u_2, \dots, u_{k-1}, u_{k} = - \sum_{i=1}^{k-1} u_i$ where $\{u_i\}_{i=1}^k$ are linearly independent is an optimal rank $k-1$ solution to the primal SDP. Furthermore, we know from \cref{result} that every primal solution has rank at most $k-1$. To prove the claim, we must show the reference solution is the only optimal solution (up to rotation). 
Consider the smallest $(k-1)$-tree, the complete graph on $k$ vertices, $K_{k}$. Choose $k-1$ of the vertices and label the corresponding vectors as $v_1, \dots, v_{k-1}$. Then we show these $k-1$ vectors are linearly independent. Assume they are not and let $j \leq k-1$ be the smallest index such that $v_j = \sum_{r=1}^{j-1} \alpha_r v_r$ for some $\alpha_1, \dots, \alpha_{j-1}$. By the constraints of the vector program, we know $v_i \cdot v_j = -1/(k-1)$ for $i = 1, \cdots, j-1$. Furthermore, for any such $i$, 
$$ -1/(k-1) = v_i \cdot v_j = v_i \cdot \left(\sum_{i=r}^{j-1} \alpha_r v_r\right) = \alpha_i - \frac{1}{k-1}\left(\sum_{r \neq i: r \in \{1,\dots,j-1\}}\alpha_r\right).$$ If we subtract two of these equations deriving from $i=i_1, i_2$, we see that $$ 0 = (\alpha_{i_1} - \frac{1}{k-1}\alpha_{i_2}) - (\alpha_{i_2} - \frac{1}{k-1}\alpha_{i_1}).$$ This implies $\alpha_{i_1} = \alpha_{i_2}$ and thus the coefficient $\alpha_i$ is the same for every $i \in \{1, \dots, j-1\}$. Relabeling this coefficient as $\alpha$, we can write $v_j = \alpha\sum_{r=1}^{j-1}v_r.$ Considering $-1/(k-1) = v_i \cdot v_j$ again, we see $-1/(k-1) = \alpha(1-\frac{j-2}{k-1}),$ so $\alpha = -\frac{1}{k-j+1}.$ However, $1 = v_j \cdot v_j$ tells us 
$$
\begin{aligned}
1 &= \alpha^2 (\sum_{r=1}^{j-1} v_r) \cdot (\sum_{r=1}^{j-1}v_r) \\ &= \alpha^2 \sum_{r=1}^{j-1}\left(v_r \cdot (\sum_{p=1}^{j-1}v_p) \right)\\
&= \alpha^2 \sum_{r=1}^{j-1}  (1 - \frac{j-2}{k-1})\\ &= \alpha^2((j-1)(1-\frac{j-2}{k-1})) \\ &= \frac{1}{(k-j+1)^2}\frac{(j-1)(k-j+1)}{k-1}\\ & = \frac{j-1}{(k-1)(k-j+1)} .
\end{aligned}
$$ 
This means $j = \frac{k^2}{k} = k$.

Now we consider the vector $v_{k}$ assigned to the final vertex of the $(k-1)$-tree. Again, we have $v_i \cdot v_{k} = -1/(k-1)$ for $i = 1, \dots, k-1$. For a given $i$, the set of solutions to $v_i \cdot v_{k}$ is represented by a hyperplane, $H_i$ in the $(k-1)$-dimensional vector space spanned by $\{v_1, \dots, v_{k-1}\}$. Therefore, a satisfying vector $v_{k}$ must lie in $H_1 \cap \dots \cap H_{k-1}$. Because $v_1, \dots, v_{k-1}$ are linearly independent, $dim(H_1 \cap \dots \cap H_{k-1}) = 0$. Thus there is a unique vector that satisfies the given equations. Since $v = -\sum_{i=1}^{k-1} v_i$ satisfies all $k-1$ equations, we find that $v_{k} = v$. Therefore, this solution is exactly the reference solution.

Now, assume the claim is true for all $(k-1)$-trees on $n$ vertices. Consider a $(k-1)$-tree, $G$, with $n+1$ vertices and a rank $k-1$ primal solution. $G$ is constructed from a $(k-1)$-tree, $G'$, with $n$ vertices by attaching an additional vertex $v$ with edges to all vertices in a $(k-1)$-clique, $K$, of $G'$. Therefore the primal solution to $G$ is also a rank $k-1$ primal solution to $G'$. Then by induction, we know this must be the reference solution $v_1, \dots, v_{k-1}, v_{k} = -\sum_{i=1}^{k-1} v_i$. By symmetry, we may assume that the vertices in $K$ are assigned to $v_1, \dots, v_{k-1}$. Finally, if $v_v$ is the vector assigned to $v$, then by the same hyperplane argument, we find that $v_v = v_{k}$. Therefore the primal solution to $G$ is again the reference solution.
\end{proof}

\cref{thm:svcn-color} shows that we can partition the vertices of a $(k-1)$-tree into $k$ sets with each set associated to a different vector assigned in the low rank primal solution. Since vertices $u,v$ are only in the same set in the partition if they were assigned the same vector in the primal solution, it is not possible for neighbors to be in the same set. We can then produce a valid coloring of the vertices by associating one color to each set in the partition. 


We now turn to characterizing cases in which we cannot find dual solutions of sufficiently high rank by looking at potential solutions of vector colorings for graphs without unique colorings. In particular, we restrict our attention to graphs that have multiple distinct $k$-colorings and contain a $k$-clique. These assumptions provide information about the optimal objective function values.

\begin{theorem}\label{multiple}
Let $G$ be a graph with $n$ vertices, multiple distinct $k$-colorings, and a $k$-clique. There exists a primal solution to the strict vector chromatic number program for $G$ with rank greater than $k-1$, and thus by Fact \ref{fact:sdp-dual-rank} the rank of any optimal dual solution must be less than $n-k+1$.
\end{theorem}

\begin{proof}
Let $c_1, c_2$ be two distinct $k$-colorings of $G$. Let $K$ be a $k$-clique in $G$. Begin with a reference solution of rank $k-1$ and assign each color class a corresponding vector in such a way that $c_1(i) = c_2(i)$ for $i \in K$. This fixes the color labelling for the vertices in $K$. Then, we can represent these colorings by the PSD matrices $C_1$ and $C_2$, respectively, where $C_p(ij) = 1$ if $c_p(i)=c_p(j)$  and $C_p(ij) = -1/(k-1)$ if $c_p(i) \neq c_p(j)$ for $p=1,2$ and $i,j \in [n]$. Note then, for $\alpha \in (0,1)$, $X= \alpha C_1 + (1-\alpha)C_2$ is also a valid solution to the SDP. It suffices to prove that $X$ has rank greater than $k-1$ for some $\alpha \in (0,1)$. 

Because $C_1$ and $C_2$ are PSD, for any value of $\alpha \in (0,1)$, $ker(X) = ker(C_1) \cap ker(C_2)$. We will show there is a vector $x \in ker(C_1)$ such that $x \notin ker(C_2)$ from which the result directly follows. 

Let $v$ be a vertex whose color changes, i.e. $c_1(v) \neq c_2(v)$.  Then $v$ cannot be in $K$. Let $s\in K$ be such that $c_1(s) = c_1(v)$ and thus $c_2(s) \neq c_2(v)$. Let $i_1, i_2, \dots, i_{k-2} \in K$ such that $c_2(i_j) \neq c_2(v)$ for $j = 1, \dots k-2$. Also note that $c_1(i_j) \neq c_1(v)$ since $c_1(s) = c_1(v)$ and $(i_j, s) \in E$ for $j=1, \dots , k-2$. Because $dim(ker(C_1)) = n-(k-1) = n-k+1$, there exists $x\in ker(C_1)$ such that $x \neq 0$ but $x(i) = 0$ for $i \neq v,s,i_1,\dots,i_{k-2}$. Assume $x \in ker(C_2)$; we show this leads to a contradiction. Then,
$$
 (C_1x)(v) = C_1(vv)x(v) + C_1(vs)x(s) + \sum_{j=1}^{k-2}C_1(vi_j)x(i_j) = x(v) + x(s) -\frac{1}{k-1}\sum_{j=1}^{k-2}x(i_j) = 0 
 $$ and
 $$
 (C_2x)(v) = C_2(vv)x(v) + C_2(vs)x(s) + \sum_{j=1}^{k-2}C_2(vi_j)x(i_j) = x(v) -\frac{1}{k-1} x(s) -\frac{1}{k-1}\sum_{j=1}^{k-2}x(i_j) = 0
$$
from which we can conclude that $x(s) = 0$. Similarly,
$$
 (C_1x)(s) = C_1(sv)x(v) + \sum_{j=1}^{k-2}C_1(si_j)x(i_j) = x(v) -\frac{1}{k-1}\sum_{j=1}^{k-2}x(i_j) = 0 
 $$ and
 $$
 (C_2x)(s) = C_2(sv)x(v) + \sum_{j=1}^{k-2}C_2(vi_j)x(i_j) = -\frac{1}{k-1} x(v) -\frac{1}{k-1}\sum_{j=1}^{k-2}x(i_j) = 0
$$
imply that $x(v) = 0$. By considering row $i_j$ for $j=1, \dots, k-2$, we see that the $x(i_j)$ satisfy $$A[x(i_1), x(i_2), \dots, x(i_{k-2})]^T = 0$$ where $A$ is the $(k-2) \times (k-2)$ matrix with 1 along the diagonal and $-1/(k-1)$ everywhere else. We can write $A$ as $A = -\frac{1}{k-1}J + \frac{k}{k-1}I$ where $J$ is the all 1s matrix. Then $A$ has eigenvalues $2/(k-1)$ with multiplicity 1 and $k/(k-1)$ with multiplicity $k-3$, and thus has trivial nullspace. Therefore  $[x(i_1), x(i_2), \dots, x(i_{k-2})]=0$ which contradicts that $x\neq 0$. Then $x \notin ker(C_2)$, so $rank(X) > k-1$.

\end{proof}

While we have shown that $(k-1)$-trees have sufficiently high dual rank for the standard vector chromatic number SDP, it would be nice if we could completely characterize which graphs have high dual rank. A reasonable guess would be that a $k$-colorable graph $G$ containing a $k$-clique has high dual rank if and only if it is uniquely colorable. This assertion is true for the important special case of planar graphs.

\begin{corollary}
A planar graph with $n$ vertices has dual rank at least $n-3$ if and only if it is uniquely colorable.
\end{corollary}

\begin{proof}
Fowler \cite{Fowler98} shows that uniquely-colorable planar graphs are exactly the set of planar 3-trees. By \cref{result} we know such graphs have dual rank at least $n-3$. Furthermore, \cref{multiple} shows that graphs with multiple colorings have primal solutions with rank more than 3 and therefore do not have dual solutions with rank $n-3$. 
\end{proof}

Unfortunately, the following example shows unique colorability is not sufficient in general for a sufficiently high dual rank. Hillar and Windfeldt \cite[Figure 2]{HillarW08} presented the uniquely 3-colorable graph in \Cref{fig:counterexample} excluding vertex 25 which adds a triangle. Computing the primal and dual SDPs of this graph returns solutions with objective value $-0.5$, primal rank of 24, and dual rank of 1. If the claim were true, we would expect all dual solutions to have rank at least 23.

Thus it remains an interesting open question to characterize in general cases in which graphs have sufficiently high dual rank and have the reference solution as the optimal primal solution. 

\begin{figure}[!h]
\scalebox{.6}{
\centering
\begin{tikzpicture}
[auto, ->,>=stealth',node distance=3cm and 3cm,semithick,
vertex/.style={circle,draw=black,thick,inner sep=0pt,minimum size=4mm}]

\node[vertex, label=left:1, fill=black!25!white] (1) at (0,0) {};
\node[vertex, label=left:2, fill=black] (2) at (0,7) {};
\node[vertex, label=right:3] (3) at (7,7) {};
\node[vertex, label=right:4] (4) at (7,0) {};
\node[vertex, label=below:5, fill=black!25!white] (5) at (2,3) {};
\node[vertex, label=above:6, fill=black] (6) at (2,4) {};
\node[vertex, label=above:7] (7) at (3,5) {};
\node[vertex, label=above:8, fill=black!25!white] (8) at (4,5) {};
\node[vertex, label=above:9, fill=black] (9) at (5,4) {};
\node[vertex, label=below:10, fill=black!25!white] (10) at (5,3) {};
\node[vertex, label=below:11, fill=black] (11) at (4,2) {};
\node[vertex, label=below:12] (12) at (3,2) {};
\node[vertex, label=right:13, fill=black] (13) at (17,0) {};
\node[vertex, label=left:14] (14) at (10,0) {};
\node[vertex, label=left:15, fill=black!25!white] (15) at (10,7) {};
\node[vertex, label=right:16, fill=black!25!white] (16) at (17,7) {};
\node[vertex, label=below:17] (17) at (14,2) {};
\node[vertex, label=below:18, fill=black!25!white] (18) at (13,2) {};
\node[vertex, label=below:19, fill=black] (19) at (12,3) {};
\node[vertex, label=above:20] (20) at (12,4) {};
\node[vertex, label=left:21, fill=black] (21) at (13,5) {};
\node[vertex, label=right:22, fill=black] (22) at (14,5) {};
\node[vertex, label=above:23] (23) at (15,4) {};
\node[vertex, label=below:24, fill=black!25!white] (24) at (15,3) {};
\node[vertex, label=right:25, fill=black] (25) at (16,3.5) {};


\draw[-] (1) -- (2);
\draw[-] (1) -- (4);
\draw[-] (1) -- (6);
\draw[-] (1) -- (12);

\draw[-] (2) -- (3);
\draw[-] (2) -- (7);
\draw[-] (2) .. controls (-6,-5) and (9,-1) .. (14);

\draw[-] (3) -- (10);
\draw[-] (3) .. controls (8,3.5) and (11,-1) .. (18);

\draw[-] (4) -- (9);
\draw[-] (4) -- (11);

\draw[-] (5) -- (6);
\draw[-] (5) -- (9);
\draw[-] (5) -- (12);

\draw[-] (6) -- (7);
\draw[-] (6) -- (10);

\draw[-] (7) -- (8);
\draw[-] (7) -- (11);
\draw[-] (7) .. controls (7,9) .. (15);

\draw[-] (8) -- (9);
\draw[-] (8) -- (12);

\draw[-] (10) -- (11);
\draw[-] (10) .. controls (9,2) .. (14);
\draw[-] (10) .. controls (11.75,6) .. (22);

\draw[-] (11) -- (12);

\draw[-] (13) -- (14);
\draw[-] (13) -- (16);
\draw[-] (13) -- (18);
\draw[-] (13) -- (24);

\draw[-] (14) -- (15);
\draw[-] (14) -- (19);

\draw[-] (15) -- (22);

\draw[-] (16) -- (21);
\draw[-] (16) -- (23);

\draw[-] (17) -- (18);
\draw[-] (17) -- (21);
\draw[-] (17) -- (24);

\draw[-] (18) -- (19);
\draw[-] (18) -- (22);

\draw[-] (19) -- (20);
\draw[-] (19) -- (23);

\draw[-] (20) -- (21);
\draw[-] (20) -- (24);

\draw[-] (22) -- (23);

\draw[-] (23) -- (24);
\draw[-] (23) -- (25);

\draw[-] (24) -- (25);

\end{tikzpicture}
}
\caption{Uniquely 3-colorable graph with a $K_3$ which does not have any dual optimal solution of sufficiently high dual rank.}
\label{fig:counterexample}
\end{figure}
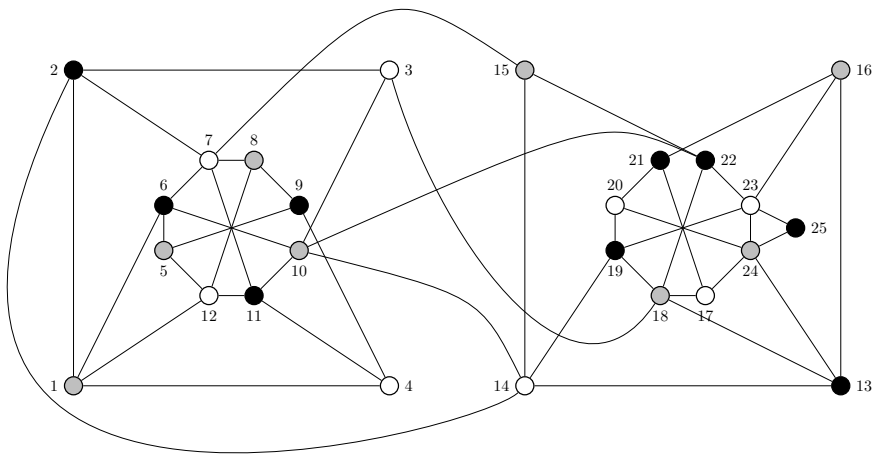

\section{A Semidefinite Program with Costs}
\label{modify}
Unfortunately, \cref{multiple} seems to indicate that this method of looking for graphs that have high dual rank with the standard vector chromatic number SDP cannot be generalized to graphs with multiple colorings. To extend this method, we consider a modified SDP described next. The new program utilizes a new objective function. Here, we introduce the notion of a cost matrix $C(G)$. The goal is to identify a $C(G)$ such that minimizing $C(G) \bullet X$ forces $X$ to have our desired rank. In particular, we consider the SDP given by 

\begin{equation*}
    \begin{array}{llll}
        \text{minimize} & C(G) \bullet X &\\
        \text{subject to} & X_{ij} = -1/(k-1), & \forall (i,j)\in E, \\
         (CP) & X_{ii} = 1,    & \forall i \in V, \\
         & X \succeq 0.& 
    \end{array}
\end{equation*}

\noindent We observe that the solutions to (SVCN) with $\alpha = -1/(k-1)$ are exactly the feasible solutions to (CP).  The corresponding dual SDP is 

\begin{equation*}
    \begin{array}{llll}
        \text{maximize} & \sum_{i=1}^n y_i - \frac{2}{k-1} \sum_{e \in E}z_e &\\
        \text{subject to} & S = C - \sum_{i=1}^n y_iE_{ii} - \sum_{e \in E} z_e E_e, &  \\
        (CD)  & S\succeq 0.& 
    \end{array}
\end{equation*}

\noindent where $E_{ii}$ is the matrix with a 1 at position $ii$ and 0 elsewhere and for $e=(i,j)$, $E_e$ is the matrix with 1 at positions $ij$ and $ji$ and 0 elsewhere.

To demonstrate how this cost matrix influences the behavior of $rank(X)$, assume that $G=(V,E)$ is a $k$-colorable graph containing a $k$-clique, but is not a $(k-1)$-tree. We still know there is a solution to the strict vector chromatic number program with $\alpha = -1/(k-1)$, and thus it is possible to find an $X$ satisfying our modified vector program constraints. Now fix $c:V \rightarrow [k]$ to be a valid $k$-coloring of $G$. With this coloring, we can define an associated matrix $C(G)$ in the following way:

\[ C(G)_{ij}= \begin{cases} 
      -1 & i<j, c(i)=c(j), \forall \ell \text{ such that } i<\ell<j, c(i) \neq c(\ell) \\
       & \\
      -1 & i>j, c(i)=c(j), \forall \ell \text{ such that } i>\ell>j, c(i) \neq c(\ell)  \\
      & \\
      0 & \text{otherwise}.
   \end{cases}
\]
Intuitively, the reference solution corresponding to the coloring given by $c$ is the solution that will minimize total cost since we'll look for a solution $X$ with $X_{ij}=1$ exactly when $C(G)_{ij}=-1$; for such entries, we'll have the same vectors corresponding to vertices $i$ and $j$.  But we can show additionally that there is a dual optimal solution for cost function $C(G)$ that has sufficiently high rank.

\begin{restatable}{theorem}{csdp}\label{dependent}
For $G$ and $C(G)$ as described, there is an optimal dual solution with rank at least $n-k+1$, so that any optimal primal solution has rank at most $k-1$.
\end{restatable}

Let $K$ be a $k$-clique in our $k$-colorable graph $G$. Let $s_i$ denote the sum of entries in column $i$ of $C(G)$. Consider the assignment of dual variables given by $y_i = s_i$ for $i \notin K$, $y_i = s_i-1$ for $i \in K$, $z_e = -1$ for $e=(i,j), i,j\in K, i\neq j$, and $z_e = 0$ otherwise. We denote this assignment by $(y,z)$.

\begin{lemma}
The dual matrix $S$ constructed with $(y,z)$ is positive semidefinite.
\end{lemma}

\begin{proof}
Consider the complete graph on $k$ vertices given by $G=K_k$ (as this is the smallest possible $k$-colorable graph containing a $k$-clique). Observe that $C(K_k)$ is the matrix of all 0s as no two vertices can be colored the same. Thus $s_i = 0$ for all $i\in K_k$. Furthermore, $(y,z)$ assigns $y_i = -1 $ for all $i \in K_k$ and $z_e = -1$ for all $e \in K_k$. Then $S$ is the all-ones matrix with eigenvalues $k$ and 0 with multiplicity $k-1$, and thus is positive semidefinite. 

Now assume the claim is true for all $k$-colorable graphs containing a $k$-clique that have at most $n$ vertices. Let $G=(V,E)$ such that $|V| = n+1$, $G$ is $k$-colorable, and $G$ contains a $k$-clique, $K$. Then $G$ can be constructed by adding a vertex $v_{n+1}$ and its adjacent edges to some graph $G'=(V', E')$ such that $|V'| = n$, $G'$ is $k$-colorable, and $G'$ contains $K$. By assumption, the matrix $S'$ corresponding to $G'$, $\{y_i'\}_{i=1}^n,$ and $\{z_e'\}_{e\in E'}$ is positive semidefinite. 

Let $v_m$ be the largest-indexed vertex that is the same color as $v_{n+1}$ after it is added to $G$. We consider how the addition of $v_{n+1}$ affects $C(G'), \{y_i\}_{i=1}^{n+1},$ and $\{z_e\}_{e\in E}.$ For $i,j \neq n+1$, $C(G)_{ij} = C(G')_{ij}$. We also observe $C(G)_{m(n+1)} = C(G)_{(n+1)m} = -1$ and $C(G)_{i(n+1)} = C(G)_{(n+1)i} = 0$ for $i\neq m$. Furthermore, for $i \neq m, n+1$, $y_i = y_i'$, while $y_m = y_m' -1$ and $y_{n+1} = -1$. Finally, $z_e = -1$ for $e \in K$ and 0 otherwise.

With this update, we see that 

$$ S(G) = \begin{bmatrix}
   & &  &\vline & 0\\
    & S'(G') & &  \vline & \vdots \\
    & & &  \vline &  0\\
\hline
  0 & \cdots & 0 & \vline & 0
\end{bmatrix} + vv^T$$ where $v^T = [0,\dots, 0, -1,0,\dots,0,1]$. It follows that $S(G)$ is positive semidefinite since $S'(G')$ is PSD by assumption and for all $x \in \mathbb{R}^{n+1}$, $x^Tvv^Tx = (v^Tx)^2 \geq 0$.

\end{proof}

\begin{theorem}\label{costdual}
The assignment $(y,z)$ is an optimal dual solution, and the reference solution is an optimal primal solution.
\end{theorem}

\begin{proof}
The previous lemma shows $(y,z)$ satisfies the constraints. Thus to prove this claim, it suffices to show $(y,z)$ maximizes the objective function. We demonstrate this by showing that the reference solution has the same objective function value.

First, we consider the dual objective function value for $(y,z).$ We have
$$
\begin{aligned}
    & \sum_{i=1}^n y_i - \frac{2}{k-1}\sum_{e\in E} z_e \\
    &= (\sum_{i=1}^n s_i) - k + \frac{2}{k-1} {k \choose 2} \\
    &= (\sum_{1 \leq i,j \leq n} C(G)_{ij}) - k + k = \sum_{1 \leq i,j \leq n} C(G)_{ij}.
\end{aligned}
$$

Now, let $X$ be the matrix given by the reference solution:
\[ X_{ij}= \begin{cases} 
      1 & i=j \\
       & \\
      1 & c(i) = c(j)  \\
      & \\
      -\frac{1}{k-1} & c(i) \neq c(j)
   \end{cases}
\]
\noindent where $c: V \rightarrow [k]$ is the fixed coloring used to generate $C(G)$. Note that $X$ satisfies the constraints of the primal SDP (CP). The objective function value is given by 

$$
\begin{aligned}
C(G) \bullet X &= \sum_{i,j: c(i) = c(j)} C(G)_{ij} - \frac{1}{k-1}\sum_{i,j: c(i) \neq c(j)} C(G)_{ij} \\
&= \sum_{i,j: c(i) = c(j)} C(G)_{ij} = \sum_{1\leq i,j, \leq n} C(G)_{ij}.
\end{aligned}
$$

\noindent Since the primal and dual objective function values are equal, the corresponding solutions must be optimal.
\end{proof}

Finally, we restate and prove \cref{dependent}.

\csdp*

\begin{proof}
Again begin by considering the complete graph on $k$ vertices given by $K_k$. As previously discussed, the matrix $S$ determined by $(y,z)$ is simply the all-ones matrix. It is straightforward to see this has rank 1 = $k-k+1$.

Assume the claim holds for all $k$-colorable graphs containing a $k$-clique with at most $n$ vertices. Let $G$ be a $k$-colorable graph containing a $k$-clique with $n+1$ vertices. Following the same decomposition used previously, we can write
$$ S(G) = T + vv^T$$ where 

$$v(i) = \begin{cases}
-1 & i = m \\
1 & i=n+1 \\
0 & otherwise
\end{cases}
$$ for $v_m$ the largest-valued vertex colored the same as $v_{n+1}$ and 
$$T = \begin{bmatrix}
   & &  &\vline & 0\\
    & S'(G') & &  \vline & \vdots \\
    & & &  \vline &  0\\
\hline
  0 & \cdots & 0 & \vline & 0
\end{bmatrix}.$$

Since $S(G)$ is a sum of positive semidefinite matrices, $ker(S(G)) = ker(T) \cap ker(vv^T)$. Observe that $x= (0,\dots, 0,1) \in ker(T)$. However, $x \notin ker(vv^T)$. Thus $dim(ker(S(G))) = dim(ker(T) \cap ker(vv^T)) < dim(ker(T)) \leq k $ as desired. 
\end{proof}

\begin{theorem}\label{uniquesol}
For $G$ and $C(G)$ as described, the reference solution is the unique optimal primal solution.
\end{theorem}

\begin{proof}
\cref{costdual} tells us that the reference solution is an optimal primal solution, while  \cref{dependent} tells us that any optimal primal solution has rank at most $k-1$. Therefore it suffices to show that any rank $k-1$ optimal primal solution is in fact the reference solution.

From the proof of \cref{costdual}, we know that the optimal objective function value is $\sum_{1\leq i,j \leq n}C(G)_{ij}.$ Furthermore, any primal feasible $X$ satisfying $C(G) \bullet X = \sum_{1\leq i,j \leq n}C(G)_{ij}$ must have $X_{ij} = 1$ whenever $C(G)_{ij} = -1$ since $X_{ij} \leq 1$ for all $1\leq i,j \leq n$ and each entry of $C(G)$ is either 0 or -1. Let $K$ be a $K_k$ in $G$, $X$ be a rank $k-1$ primal optimal solution and $c$ be the $k$-coloring used to construct $C(G)$. Define $k$ sets by $S_i = \{v \in V: c(v) = i\}$ for $i=1,\dots, k$ and let $n_i = |S_i|.$ 

We claim that $X$ assigns each vertex a vector from the reference solution depending only on which $S_i$ the vertex is a member of. For $i=1, \dots, k$, sort the vertices in $S_i$ from smallest label to largest so that $S_i = \{v_{i_1}, v_{i_2}, \dots, v_{i_{n_i}}\}$ where $i_1 < i_2 < \dots < i_{n_i}.$ By construction of $C$, we have that $C_{i_1i_2} = C_{i_2i_3} = \dots = C_{i_{n_i-1}i_{n_i}} = -1$ implying $X_{i_1i_2} = X_{i_2i_3} = \dots = X_{i_{n_i-1}i_{n_i}} = 1$. In particular, $v_{i_1}, v_{i_2}, \dots, v_{i_{n_i}}$ are assigned the same unit vector by $X$.

Let $u_i$ be the vector corresponding to $S_i$ for $i=1, \dots, k$. We know that exactly one member of $K$ must be in each $S_i$ for $i=1,\dots,k$, so $u_i \cdot u_j = -1/(k-1)$ for $1 \leq i, j \leq k$. Then following the proof of \cref{thm:svcn-color}, we see that these vectors are exactly the reference solution.
\end{proof}

\section{Experimental Results}\label{experiment}

Two heuristics have been implemented and experimentally demonstrated success returning low-rank primal solutions for planar graphs. Neither algorithm assumes knowledge of a graph coloring.  We tested these heuristics on all maximal planar graphs of up to 14 vertices that contain a $K_4$.  These graphs were generated  via the planar graph generator plantri due to Brinkmann and McKay \cite{BrinkmannM07} found at \url{https://users.cecs.anu.edu.au/~bdm/plantri/}. The `-a' switch was used to produce graphs written in ascii format. The code was implemented in Python using the MOSEK Optimizer as the SDP solver. Both the graph data files and algorithm implementation can be found at \url{https://github.com/rmirka/four-coloring.git}.  Our results are shown in Table \ref{table:expresults}.  The heuristics successfully colored all graphs with up to 11 vertices, and successfully colored 99.75\% of the graphs of 12-14 vertices.  We do not record the running time of the heuristics; because the heuristics involve repeatedly solving semidefinite programs, they are not competitive with other greedy or local search style heuristics.  Our primary reason for studying these heuristics was to find whether we could reliably find a cost matrix $C$ giving rise to a four-coloring for planar graphs.

\begin{table}[h!]
\centering
\begin{tabular}{ |c|c|c|c| } 
 \hline
\# nodes & \# maximally planar graphs with $K_4$ & \# heuristic 1 failures & \# heuristic 2 failures\\ 
 \hline
5 & 1 & 0 & 0 \\ 
6 & 1 & 0 & 0 \\ 
 7 & 4 & 0 & 0 \\ 
 8 & 12 & 0 & 0 \\ 
 9 & 45 & 0 & 0 \\ 
10 & 222 & 0 & 0 \\ 
11 & 1219 & 0 & 0 \\ 
12 & 7485 & 18 ($\sim$ .24\%) & 18 ($\sim$ .24\%)\\ 
13 & 49149 & 108 ($\sim$ .22\%) & 116 ($\sim$ .24\%) \\
14 & 337849 & 619 ($\sim$ .18\%) & 811 ($\sim$ .24\%) \\
 \hline
\end{tabular}
\caption{This table depicts the number of times the heuristic algorithms failed on maximally planar graphs with between 5 and 14 vertices.}
\label{table:expresults}
\end{table}


At a high-level, both heuristics follow the same procedure. At each step, they solve the vector program (CP) given in \cref{modify}. If the returned solution does not have the desired rank, the cost matrix $C$ is updated and the process is repeated. The heuristics differ in how the cost matrix is updated. 


\begin{algorithm}[]
\SetAlgoLined
 Find a clique $K = \{k_1,k_2,k_3,k_4\}$ \;
 $C = 0$\;
 $i,j=0$\;
 $badcolors = []$\;
 $good = True$\;
 $X,S,r,p = solveModified(G,C)$\Comment*[r]{call the Mosek optimizer to solve the SDP on graph G with cost matrix C and return the primal and dual matrices ($X,S$, respectively) and ranks($r,p$, respectively)}
 \While{$r>3$}{
 \If{$X_{ik_j} \neq 1$ and $good$}{
    $badcolors = badcolors \cup \{k_j\}$\;
    $C_{iS_j[length(S_j) - 2]} = C_{S_j[length(S_j) - 2]i} = 0$\;
    $X,S,r,p = solveModified(G,C)$\;
    $good = False$
 }
 \Else{
 $good = True$\;
  $S_s = \{v \in V: X_{vk_s} = 1\}$, $s=1,2,3,4$\;
  $colored = \cup_{s=1}^4 S_s$\;
  $stillLooking = True$\;
  \While{stillLooking}{
  \For{$q=1,\dots, 4$}{
    \If{$stillLooking$ and $i \notin colored$ and $k_q \notin badcolors$ and $X_{ik_q} \neq 1, -1/3$}{
    $S_q = S_q \cup \{i\}$\;
    stillLooking = False\;
    $j = q$\;
    }
  }
  \If{$stillLooking$}{
  $i = i+1 \mod n$\;
  $badcolors = []$}
  
  }
  $C=0$\;
  \For{$q=1,\dots,4$}{
    \For{$s=1,\dots,length(S_q)-1$}{
        $C_{s,s+1}= C_{s+1,s}= -1$\;}
  }
  $X,S,r,p = solveModified(G,C)$
 }
 }
 \caption{$(G=(V,E))$}
\end{algorithm}

The first heuristic (Algorithm 1) is based on the coloring-dependent cost matrix discussed in \Cref{modify}. The algorithm first identifies a $K_4 = \{k_1, k_2, k_3, k_4\}$ and finds an initial solution with $C=0$. If the primal solution does not have low enough rank, the returned solution is used to update the cost matrix. Let $S_i = \{v \in V : X_{vk_i} = 1\}$ for $i=1,2,3,4$. Let $v$ be a vertex in $V \setminus (\cup_{i=1}^4 S_i)$. Then there must exist $i^* \in \{1,2,3,4\}$ such that $X_{vk_{i^*}} \neq 1$ and $X_{vk_{i^*}} \neq -1/3$; we update this $S_{i^*}$ by adding $v$ to it. Now, $C$ is constructed based on the $S_j, j=1,2,3,4$. In particular, for $i=1,2,3,4$, if $n_i$ denotes the number of vertices in $S_i$, then for $j=1, \dots, n_i-1$, we set $C_{rs} = C_{sr} = -1$ where $r$ and $s$ are the $j$th and $j+1$st vertices in $S_i$. This new cost matrix $C$ is used to compute an updated solution $\hat{X}$. If $\hat{X}$ is of the desired rank, the algorithm terminates. If not, we first check to see if $\hat{X}_{vk_{i^*}} = 1$, i.e. if our selected vertex from the previous iteration was successfully colored. If yes, we repeat the process beginning with our solution $\hat{X}$ and selecting a currently uncolored vertex. If $v$ was not successfully colored, we remove the entry in the cost matrix corresponding to this assignment from the previous iteration and resolve the SDP while adding $k_{i^*}$ to a list of `bad' colors for $v$. We now repeat the process by selecting a new feasible color class for $v$ (following the same rules as previously in addition to requiring it not be in the list of `bad' colors for $v$) and constructing $S_i$, $i=1,2,3,4$ and $C$ accordingly.

\begin{algorithm}[]
\SetAlgoLined
 Find a clique $K = \{k_1,k_2,k_3,k_4\}$ \;
 $C = 0$\;
 $X,S,r,p = solveModified(G,C)$\Comment*[r]{call the Mosek optimizer to solve the SDP on graph G with cost matrix C and return the primal and dual matrices ($X,S$, respectively) and ranks($r,p$, respectively)}
 $v^*=k^* = k_1$\;
 $badcolors = []$\;
 $good = True$\;
 \While{$r>3$}{
    \eIf{$X_{v^*,k^*} \neq 1$ $\&\&$ good}{
    $C_{v^*,k^*} = C_{k^*,v^*} = 0$\;
    $X,S,r,p = solveModified(G,C)$\;
    $badcolors = badcolors \cup \{k^*\}$\;
    $good = False$\;
    }{
    $good = True$\;
  $S_s = \{v \in V: X_{vk_s} = 1\}$, $s=1,2,3,4$\;
  $colored = \cup_{s=1}^4 S_s$\;
  $stillLooking = True$\;
  \While{stillLooking}{
  \For{$q=1,\dots, 4$}{
    \If{$stillLooking$ and $v^* \notin colored$ and $k_q \notin badcolors$ and $X_{v^*k_q} \neq 1, -1/3$}{
    $C_{v^*k_q} = C_{k_qv^*} = -1$\;
    stillLooking = False\;
    $k^* = k_q$\;
    }
  }
  \If{$stillLooking$}{
  $v^* = v^*+1 \mod n$\;
  $badcolors = []$}
  
  }
  $X,S,r,p = solveModified(G,C)$
 }
 }
 \caption{$(G=(V,E))$}
\end{algorithm}

The second heuristic (Algorithm 2) is motivated by similar ideas but distinct cost-matrix updates. Again, the algorithm first identifies a $K_4 = \{k_1, k_2, k_3, k_4\}$ and finds an initial solution with $C=0$. Now, if $X$ is a primal solution with greater rank than desired, let $S = \{v \in V : \exists k \in K_4 \text{ such that } X_{vk} = 1\}.$ Intuitively, $S$ is the set of vertices that are aligned with the four vectors corresponding to the $K_4$ and thus have rank 3. Now, choose a single vertex $v \in V\setminus S$. Again, there must exist $k \in K_4$ such that $X_{vk} \neq 1$ or $-1/3$, so $C$ is updated such that $C_{vk} = C_{kv} = -1$. Now the vector program is run again, and the value of $X_{vk}$ in the new solution is immediately checked. If $X_{vk} = 1$ now, the algorithm proceeds as usual. However if $X_{vk} \neq 1$, the cost matrix is updated so that $C_{vk} = 0$ again and a different entry is chosen to update. Again, this process is repeated until the desired primal rank is achieved. 

In both heuristics, the termination condition is that the primal rank is equal to 3, but this doesn't necessarily guarantee that the dual rank is $n-3$. If instead one wanted to guarantee high dual rank, one could run the algorithm one more time, i.e.\ once the low-rank primal solution is achieved, extract the coloring and construct the corresponding $C$ matrix as previously described in \Cref{dependent}.

The example in \Cref{fig:obstacle} causes both heuristics to fail without coloring the graph. First we note the $K_4 = \{2,5,6,7\}$. In the first iteration of the heuristic, these are the only four vertices that are assigned colors. In the second iteration, both heuristics successfully color vertex 1 to match vertex 6. However, afterwards each heuristic is unable to color any more vertices (it tries and fails on all other possible colors for the remaining vertices).

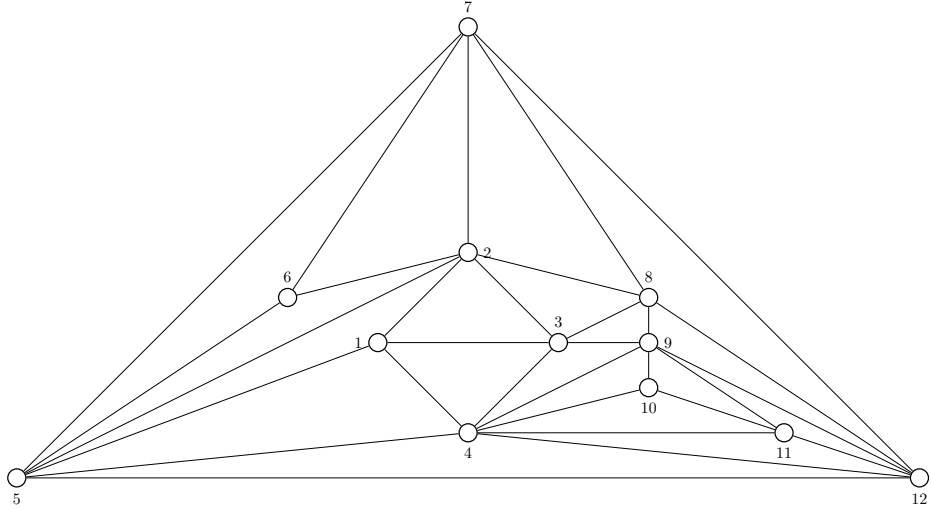
\begin{figure}[!h]
\begin{center}
\scalebox{.6}{
\centering
\begin{tikzpicture}
[auto, ->,>=stealth',node distance=3cm and 3cm,semithick,
vertex/.style={circle,draw=black,thick,inner sep=0pt,minimum size=4mm}]

\node[vertex, label=left:1] (1) at (8,3) {};
\node[vertex, label=right:2] (2) at (10,5) {};
\node[vertex, label=above:3] (3) at (12,3) {};
\node[vertex, label=below:4] (4) at (10,1) {};
\node[vertex, label=below:5] (5) at (0,0) {};
\node[vertex, label=above:6] (6) at (6,4) {};
\node[vertex, label=above:7] (7) at (10,10) {};
\node[vertex, label=above:8] (8) at (14,4) {};
\node[vertex, label=right:9] (9) at (14,3) {};
\node[vertex, label=below:10] (10) at (14,2) {};
\node[vertex, label=below:11] (11) at (17,1) {};
\node[vertex, label=below:12] (12) at (20,0) {};


\draw[-] (1) -- (2);
\draw[-] (1) -- (3);
\draw[-] (1) -- (4);
\draw[-] (1) -- (5);

\draw[-] (2) -- (3);
\draw[-] (2) -- (5);
\draw[-] (2) -- (6);
\draw[-] (2) -- (7);
\draw[-] (2) -- (8);

\draw[-] (3) -- (4);
\draw[-] (3) -- (8);
\draw[-] (3) -- (9);

\draw[-] (4) -- (5);
\draw[-] (4) -- (9);
\draw[-] (4) -- (10);
\draw[-] (4) -- (11);
\draw[-] (4) -- (12);

\draw[-] (5) -- (6);
\draw[-] (5) -- (7);
\draw[-] (5) -- (12);

\draw[-] (6) -- (7);

\draw[-] (7) -- (8);
\draw[-] (7) -- (12);

\draw[-] (8) -- (9);
\draw[-] (8) -- (12);

\draw[-] (9) -- (10);
\draw[-] (9) -- (11);
\draw[-] (9) -- (12);

\draw[-] (10) -- (11);

\draw[-] (11) -- (12);

\end{tikzpicture}
}
\end{center}
\caption{Algorithm Obstacle: $K_4 = \{2,5,6,7\}$}
\label{fig:obstacle}
\end{figure}

We considered whether our heuristics get stuck on graphs that also contained vertices resulting in irrevocable Kempe chain tangles.  Irrevocable Kempe chain tangles occur when Kempe’s local-search method of recoloring Kempe chains fails to make a color available for the vertex of interest; see Gethner et al. \cite{Gethner09} for a computational and empirical analysis of Kempe's method and irrevocable Kempe chain tangles. As such, finding an irrevocable Kempe chain tangle in a graph indicates that Kempe’s method will fail to color the graph. We tested two graphs known to contain vertices that often result in irrevocable Kempe chain tangles and slightly modified them to ensure they contained a $K_4$. The graphs are given in \Cref{fig:kempe1} and \Cref{fig:kempe2}. Our heuristics did successfully color these graphs, indicating that the class of graphs for which our algorithm does not terminate is different than the ones for which coloring with Kempe chains does not work. 

\begin{figure}[!h]
\begin{center}

\scalebox{.6}{
\centering
\begin{tikzpicture}
[auto, ->,>=stealth',node distance=3cm and 3cm,semithick,
vertex/.style={circle,draw=black,thick,inner sep=0pt,minimum size=4mm}]

\node[vertex, label=left:7] (7) at (8,3) {};
\node[vertex, label=right:6] (6) at (10,5) {};
\node[vertex, label=right:8] (8) at (12,3) {};
\node[vertex, label=below:10] (10) at (10,1) {};
\node[vertex, label=below:9] (9) at (0,0) {};
\node[vertex, label=above:2] (2) at (8,6) {};
\node[vertex, label=above:1] (1) at (10,10) {};
\node[vertex, label=above:3] (3) at (12,6) {};
\node[vertex, label=right:5] (5) at (22,2) {};
\node[vertex, label=above:4] (4) at (22,4) {};
\node[vertex, label=below:12] (12) at (24,0) {};
\node[vertex, label=below:11] (11) at (20,0) {};


\draw[-] (1) -- (2);
\draw[-] (1) -- (3);
\draw[-] (1) -- (6);
\draw[-] (1) -- (9);
\draw[-] (1) -- (11);

\draw[-] (2) -- (6);
\draw[-] (2) -- (7);
\draw[-] (2) -- (9);

\draw[-] (3) -- (6);
\draw[-] (3) -- (8);
\draw[-] (3) -- (11);

\draw[-] (4) -- (5);
\draw[-] (4) -- (11);
\draw[-] (4) -- (12);

\draw[-] (5) -- (11);
\draw[-] (5) -- (12);

\draw[-] (6) -- (7);
\draw[-] (6) -- (8);

\draw[-] (7) -- (8);
\draw[-] (7) -- (9);
\draw[-] (7) -- (10);

\draw[-] (8) -- (10);
\draw[-] (8) -- (11);

\draw[-] (9) -- (10);
\draw[-] (9) -- (11);

\draw[-] (10) -- (11);

\draw[-] (11) -- (12);

\end{tikzpicture}
}
    
\end{center}
\caption{A graph for which at least one vertex results in an irrevocable Kempe chain tangle for at least one labeling.}
\label{fig:kempe1}
\end{figure}
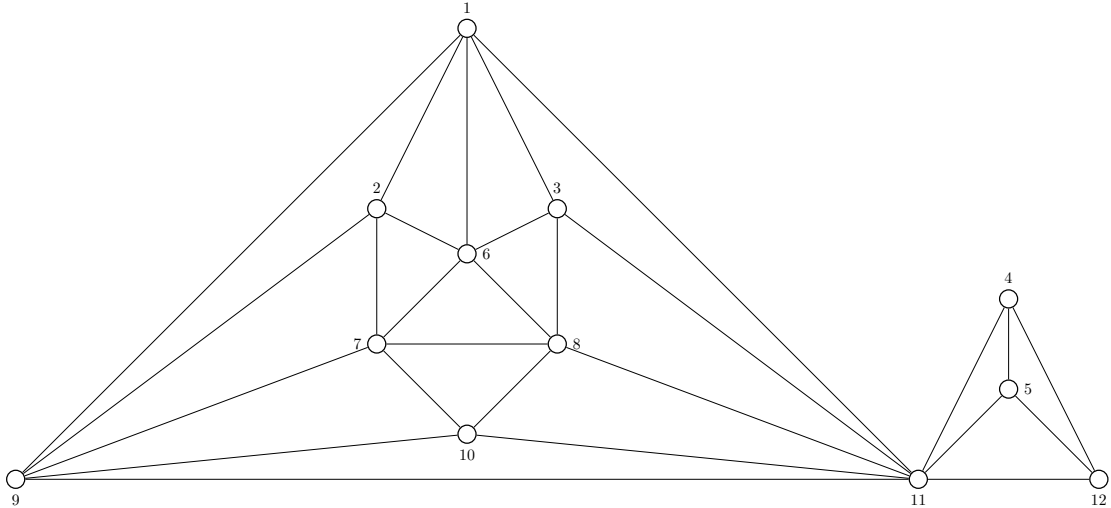

\begin{figure}[!h]
\begin{center}

\scalebox{.6}{
\centering
\begin{tikzpicture}
[auto, ->,>=stealth',node distance=3cm and 3cm,semithick,
vertex/.style={circle,draw=black,thick,inner sep=0pt,minimum size=4mm}]

\node[vertex, label=left:1] (a) at (0,0) {};
\node[vertex, label=above:2] (b) at (0,8) {};
\node[vertex, label=right:3] (c) at (2,4) {};
\node[vertex, label=below:4] (d) at (3,2) {};
\node[vertex, label=below:5] (e) at (4,6) {};
\node[vertex, label=above:6] (f) at (5,3) {};
\node[vertex, label=above:7] (g) at (8,8) {};
\node[vertex, label=above:8] (h) at (6,5) {};
\node[vertex, label=right:9] (i) at (10,12) {};
\node[vertex, label=right:10] (j) at (8,0) {};
\node[vertex, label=below:11] (k) at (10,10) {};
\node[vertex, label=below:12] (l) at (12,8) {};


\draw[-] (a) -- (b);
\draw[-] (a) -- (c);
\draw[-] (a) -- (d);
\draw[-] (a) -- (j);

\draw[-] (b) -- (c);
\draw[-] (b) -- (e);
\draw[-] (b) -- (g);

\draw[-] (c) -- (d);
\draw[-] (c) -- (e);
\draw[-] (c) -- (f);

\draw[-] (d) -- (f);
\draw[-] (d) -- (j);

\draw[-] (e) -- (f);
\draw[-] (e) -- (g);
\draw[-] (e) -- (h);

\draw[-] (f) -- (h);
\draw[-] (f) -- (j);

\draw[-] (g) -- (h);
\draw[-] (g) -- (i);
\draw[-] (g) -- (j);
\draw[-] (g) -- (k);
\draw[-] (g) -- (l);

\draw[-] (h) -- (j);

\draw[-] (i) -- (k);
\draw[-] (i) -- (l);


\draw[-] (k) -- (l);

\end{tikzpicture}
}
    
\end{center}
\caption{A second graph for which at least one vertex results in an irrevocable Kempe chain tangle for at least one labeling.}
\label{fig:kempe2}
\end{figure}
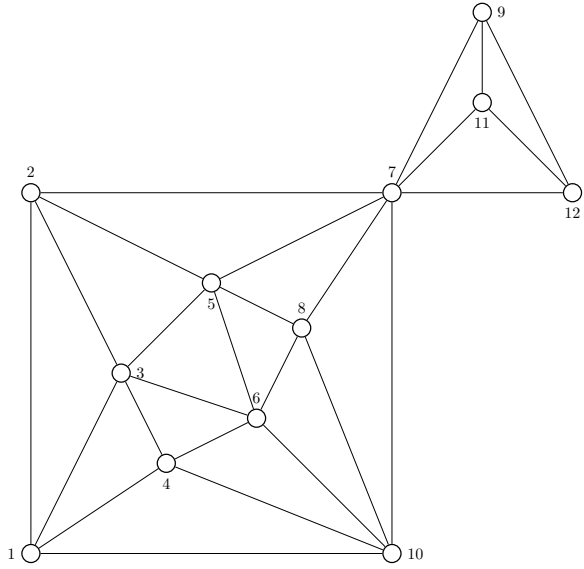

\section{Further Thoughts, Open Questions, and Conclusions}
\label{conc}

In this section, we offer further thoughts about the connection of our work with the Colin de Verdi\`ere graph parameter, give possible strengthening of our results, and conclude by posing some open questions.

\subsection{Connections with the Colin de Verdi\'ere Graph Parameter}
\label{cdv}

As mentioned in the introduction, the research in this paper was prompted by the Colin de Verdi\`ere graph parameter \cite{CDV90}.  Recall that a generalized Laplacian $L = (\ell_{ij})$ of graph $G$ is a matrix such that the entries $\ell_{ij} < 0$ when $(i,j) \in E$, and $\ell_{ij} = 0$ when $(i,j) \notin E$.  We repeat the definition of the Colin de Verdi\'ere invariant, $\mu(G)$, here.
\begin{definition}  The Colin de Verdi\`ere invariant $\mu(G)$ is the largest corank of a generalized Laplacian $L$ of $G$ such that:
\begin{enumerate}
    \item $L$ has exactly one negative eigenvalue of multiplicity one;
    \item there is no nonzero matrix $X=(x_{ij})$ such that $LX=0$ and such that $x_{ij} = 0$ whenever $i=j$ or $\ell_{ij} \neq 0$.
\end{enumerate}
\end{definition}
\noindent Recall that Colin de Verdi\'ere shows that  $\mu(G) \leq 3$ if and only if $G$ is planar; in other words, {\em any} generalized Laplacian of  a planar graph $G$ with exactly one negative eigenvalue of multiplicity 1 will have rank at least $n-3$ (modulo the second condition on the invariant). Recall further that Colin de Verdi\`ere \cite{CDV90} conjectures that $\chi(G) \leq \mu(G) + 1$; this result is known to hold for $\mu(G) \leq 4$.

As discussed in the introduction, there is an intriguing connection between the Colin de Verdi\`ere parameter and the dual slack matrix (CD), in that $S-C$ is a generalized Laplacian as long as all $z_e > 0$.  We had hoped to make use of this fact to be able to construct duals of sufficiently high rank in the case of planar graphs; however, we have been unable to see how to do so.  Recall that an optimal dual solution of corank 3 implies a primal solution with vectors in 3 dimensions for (CP), potentially implying the reference solution for a 4-coloring.

\subsection{Coloring-Independent Cost Matrix}\label{independentsection}
The method given in \Cref{modify} has a significant impediment. The $C$ matrix defined previously assumes knowledge of a coloring for a graph. Ideally, for this method to have greater impact, we would like to find a definition of $C(G)$ and a corresponding dual assignment $(y,z)$ based solely on the structure of an input graph and independent of a specific coloring, but still requiring the primal solution to be our desired low-rank solution. 

Fortunately, there is a formal way of thinking about what any possible $C(G)$ must look like. Let us again assume for a moment that $G$ is a $k$-colorable graph containing a $k$-clique and define $X$ based on a specific coloring, $c$, of $G$ as described above. Then if $S$ is an optimal dual slack matrix and $X$ is an optimal primal matrix, $XS=0$. Then for any $i,j \in [n]$, $\sum_{p=1}^n X_{ip}S_{pj}= 0$. If $X$ is the reference solution for the coloring $c$, this implies $\sum_{p=1}^n X_{ip}S_{pj}=\sum_{p:c(p)=c(i)}S_{pj} - \frac{1}{k-1}\sum_{p: c(p) \neq c(i)}S_{pj}$ = 0. In particular, let $r_1, \dots r_k$ be representatives of the $k$ color classes. Since the above is true for any $i,j$, fixing $j$ and iterating through  $i=r_1, r_2,\dots r_k$ shows $\sum_{p:c(p)=c(r_1)}S_{pj} = \dots = \sum_{p:c(p)=c(r_k)}S_{pj}$. 

This is slightly problematic as it seems to indicate either $C(G)$ or $(y,z)$ will require knowledge of a specific coloring to guarantee this relationship. However, we can at least say something using the fact that a graph containing a $k$-clique must use $k$ different colors for these vertices alone. We show below a cost matrix $C$ and an optimal dual solution for which any feasible primal solution (including the reference solution) is optimal.  

For a graph $G$ containing a $k$-clique, consider $C(G)$ given by $C(G)_{ij} = 1$ if $i=j$ or $(i,j) \in E$ and 0 otherwise. Denote the number of $K_k$s in $G$ by $\mathcal{K}$, the number of $K_k$s containing $i \in V$ by $k_i$, and the number of $K_k$s containing $(i,j) \in E$ by $k_{ij}$. Finally, for the assignment $(y,z),$ set $y_i = C(G)_{ii}-k_i$ and $z_{ij} = C(G)_{ij} - k_{ij}$.

Recall that the dual matrix $S$ is given by $S = C(G) - \sum_{i \in V} y_i E_{ii} - \sum_{e=(i,j) \in E} z_{ij} E_{e}$. Thus using $(y,z)$, $S_{ij} = k_{ij}$ for $i\neq j$ while $S_{ii} = k_i$.

\begin{lemma}
The matrix $S$ obtained using the assignment $(y,z)$ is positive semidefinite. 
\end{lemma}

\begin{proof}
Assume $G$ has only one $K_k$ composed from the vertices $v_1, \dots, v_k$. Then $S$ has one eigenvalue of $k$ with corresponding eigenvector $x_k(i) = 1$ if $i \in \{v_1, \dots, v_k\}$ and 0 otherwise. $S$ also has 0 as an eigenvalue with multiplicity $n-1$ corresponding to $n-k$ elementary unit vectors $\{e_i : i \notin K_k\}$ and $k-1$ basis vectors for the set $\{x: x_{v_1} + \dots + x_{v_k} = 0\}$. Therefore $S$ is PSD. Now if $G$ contains $p$ $K_k$s, we can write $S$ as a sum of $p$ PSD matrices where each corresponds to one of $G$'s $K_k$s. Thus $S$ is PSD. 
\end{proof}

\begin{lemma}
The matrix $S$ obtained using the assignment $(y,z)$ is optimal.
\end{lemma}

\begin{proof}
Consider the dual objective function. We have

$$
\begin{aligned}
\sum_{i\in V}y_i - \frac{2}{k-1}\sum_{(i,j) \in E }z_{ij} &= \sum_{i\in V}(C(G)_{ii}-k_i) - \frac{2}{k-1}\sum_{(i,j) \in E}(C(G)_{ij} - k_{ij}) \\
&= (\sum_{i \in V} C(G)_{ii}) -k\mathcal{K}-\left(\frac{2}{k-1}\sum_{(i,j) \in E}C(G)_{ij}\right) + \frac{2}{k-1}{k \choose 2}\mathcal{K} \\
&= \sum_{i \in V} C(G)_{ii} - \frac{2}{k-1}\sum_{(i,j)\in E}C(G)_{ij} \\
&= C(G) \bullet X
\end{aligned}
$$

\noindent for any primal feasible $X$.

\end{proof}

\subsection{Open Questions}

We close with several open questions.  We were unable to give a complete characterization of the $k$-colorable graphs with a $K_k$ for which the strict vector chromatic number (SVCN) has a unique primal solution of the reference solution.  Such graphs must be uniquely colorable, but clearly some further restriction is needed.

When we know the coloring, we can produce a cost matrix $C$ for the semidefinite program (CP) such that the reference solution is the unique optimal solution and it must have rank $k-1$.  We wondered whether one could use (CP) in a greedy coloring scheme, by incrementally constructing the matrix $C$; the graph in Figure \ref{fig:counterexample} shows that our desired scheme does not work in a straightforward manner.  Possibly one could consider an algorithm with a limited amount of backtracking, as long as one could show that the algorithm continued to make progress against some metric.

Another open question is whether one can somehow directly produce a cost matrix $C$ leading to a dual solution of sufficiently high rank that does not need knowledge of the coloring.  And we conclude with the open question that first motivated this work: is it possible to use the Colin de Verdi\`ere parameter to produce this matrix $C$?



\printbibliography

\end{document}